\documentclass[a4paper, 11pt]{article}
\usepackage{amsmath,amssymb,esint,amscd,xspace,fancyhdr,color,authblk,srcltx,fontenc,bbm}
\usepackage{hyperref}

\newtheorem{theorem}{Theorem}[section]

\newtheorem{corollary}[theorem]{Corollary}

\newtheorem{lemma}[theorem]{Lemma}

\newtheorem{remark}[theorem]{Remark}
\newenvironment{proof}[1][Proof]{\textbf{#1.} }{\hfill\rule{0.5em}{0.5em}}
{\catcode`\@=11\global\let\AddToReset=\@addtoreset
\AddToReset{equation}{section}

\AddToReset{theorem}{section}

\title{An endpoint case of the renormalization property for the relativistic Vlasov-Maxwell system}
\author{Thanh-Nhan Nguyen$^1$, Minh-Phuong Tran$^2$  \\ {\small   $^1$Department of Mathematics,  Ho Chi Minh City University of Education, \\ Ho Chi Minh City, Viet Nam\medskip\\ Email: \texttt{nhannt@hcmue.edu.vn} \medskip\\ $^2$Applied Analysis Research Group, Faculty of Mathematics and Statistics,\\ Ton Duc Thang University, Ho Chi Minh City, Viet Nam \medskip\\ Email: \texttt{tranminhphuong@tdtu.edu.vn}} }
\date{\today}

\begin{document}
 
\maketitle
\begin{abstract}

The aim of this paper is to improve the previous work on the relativistic Vlasov-Maxwell system, one of the most important equations in plasma physics. Recently in [Onsager type conjecture and renormalized solutions for the relativistic Vlasov-Maxwell system, \textit{Quart. Appl. Math.}, \textbf{78}, 193-217 (2020)], C. Bardos \textit{et al.} presented a proof of an Onsager type conjecture on renormalization property and the entropy conservation laws for the relativistic Vlasov-Maxwell system. Particularly, authors proved that if the distribution function $u \in L^{\infty}(0,T;W^{\theta,p}(\mathbb{R}^6))$ and the electromagnetic field $E,B \in L^{\infty}(0,T;W^{\kappa,q}(\mathbb{R}^3))$, with $\theta, \kappa \in (0,1)$ such that $\theta\kappa + \kappa + 3\theta - 1>0$ and $1/p+1/q\le 1$, then the renormalization property and entropy conservation laws hold. To determine a complete proof of this work, in the present paper we improve their results under a weaker regularity assumptions for weak solution to the relativistic Vlasov-Maxwell equations. More precisely, we show that under the similar hypotheses, the renormalization property and entropy conservation laws for the weak solution to the relativistic Vlasov-Maxwell system even hold for the endpoint case $\theta\kappa + \kappa + 3\theta - 1 = 0$. Our proof is based on the better estimations on regularization operators.

\medskip

\noindent 

\medskip


\end{abstract}   
                  
\section{Introduction} \label{sec:intro}
In recent years, mathematicians have devoted much attention to the relativistic Vlasov-Maxwell system, the most important equation describes the distribution of particles in phase space of a monocharged plasma under relativistic effects. There has been an increasing activity that studied the Vlasov-Maxwell system in kinetic plasma physics. It is well-known that the Vlasov equation describes the time evolution of particles in a plasma, how the plasma response to electromagnetic fields. This equation finds the unknown distribution function of particles $u=u(t,x,\varsigma)$ satisfies: 
\begin{align}\label{eq:sys1}
\partial_t u + v \cdot \nabla_x u + \mathcal{F} \cdot \nabla_{\varsigma} u = 0, 
\end{align}
where $(t, x, \varsigma) \in \mathbb{R}^+\times \mathbb{R}^3 \times \mathbb{R}^3$ represent time, position and momentum of particles, respectively. The relativistic velocity $v$ of a particle with momentum $\varsigma \in \mathbb{R}^3$ is given by
\begin{align}\label{eq:sys4}
v(\varsigma) =   \frac{\varsigma}{\sqrt{1+|\varsigma|^2}}.
\end{align}
The consideration of problem may be under electromagnetic, in which the Lorentz force $\mathcal{F} = E + v \times B$ corresponds to the self-consistent electric field $E = E(t,x)$ and magnetic field $B=B(t,x)$ generated by the charged particles in the plasma. They are coupled satisfying Maxwell's equations
\begin{align}\label{eq:sys2}
\partial_t E - \mbox{curl}\, B = -j, & \quad \partial_t B + \mbox{curl}\, E = 0,\\ \label{eq:sys2b}
\mbox{div}\, E = \rho, & \quad \mbox{div}\, B = 0,
\end{align}
where the quantities $\rho=\rho(t,x)$ and $j=j(t,x)$ are the charge density and electric current density of the plasma, respectively, defined by
\begin{align}\label{eq:sys3}
\rho(t,x) = \int_{\mathbb{R}^3} u(t,x,\varsigma)d\varsigma; \quad j(t,x) = \int_{\mathbb{R}^3}v(\varsigma)u(t,x,\varsigma)d\varsigma.
\end{align}

Maxwell's equations must be solved together with the Vlasov equation \eqref{eq:sys1}, so-called the Vlasov-Maxwell system. Here, we are interested in the Cauchy problem for system \eqref{eq:sys1}-\eqref{eq:sys3}, where the initial data given as
\begin{align*}
&u(0,x,\varsigma) = u_0(x,\varsigma) \ge 0,\\ 
&E(0,x) = E_0(x), \quad B(0,x) = B_0(x), \\ 
&\mbox{div}\, E_0 = \rho_0 = \int_{\mathbb{R}^3}u_0d\varsigma, \quad \mbox{div}\, B_0 = 0.
\end{align*}

There are many interesting problems that related to the Vlasov-Maxwell system~\eqref{eq:sys1}-\eqref{eq:sys3} that make the range of its application has been considerably extended. For instance, the existence and uniqueness of analytical solutions to this, especially for high dimensions; regularity results for the system in some spaces; the conduction of sharp estimates for solutions; some numerical methods and simulations on the solutions, etc, are at the core of many researching topics at the moment. 

The global existence of solution to this earlier has been studied intensively by several authors, such as R.J.~Diperna and P.-L.~Lions in~\cite{Diperna}, Y. Guo in~\cite{Guo, Guo2} or G. Rein in~\cite{Rein} and several references therein. Later, different approaches to the results related to this system were recently achieved and reviewed by other authors. In our knowledge, there has been a few results on the regularity of this system. Recently in 2018, N. Besse \textit{et al.} have showed in~\cite{Besse} that if the macroscopic kinetic energy is in $L^2$, then the electric and magnetic fields belong to the Sobolev space $H^s_{\text{loc}}(\mathbb{R}^+\times \mathbb{R}^3)$ with $s = 6/(13+\sqrt{142})$. Moreover, in~\cite{Bardos2,Gwiazda}, authors have established the critical regularity of weak solutions to a general system of entropy conservation laws which are related to the famous Onsager exponent 1/3. In the nearest research paper~\cite{Bardos}, Bardos {\it{et al.}} gave a proof of an Onsager type conjecture on renormalization property and entropy conservation laws for the Vlasov-Maxwell equations. More precisely, their work was devoted to the results that if the distribution function $u \in L^{\infty}(0,T;W^{\theta,p}(\mathbb{R}^6))$ and the electromagnetic fields $E,B \in L^{\infty}(0,T;W^{\kappa,q}(\mathbb{R}^3))$, where $\theta, \kappa \in (0,1)$ satisfying $\theta\kappa + \kappa + 3\theta - 1>0$ with $1/p+1/q\le 1$, then the renormalization property holds. As there have been very few results concerning the regularity of this system, such extensions have been promising to discuss under various assumptions and conditions of the problem formulation. In the present paper, based on the regularity assumptions of weak solution to the Vlasov-Maxwell equations, a small portion of that result is improved, where the conclusion of this property holds even for $\theta\kappa + \kappa + 3\theta - 1 \ge 0$, the renormalization property and entropy conservation laws hold under the same hypotheses. To our knowledge, from the mathematical point of view, for the endpoint case $\theta\kappa + \kappa + 3\theta - 1 = 0$, the proof is more challenging than what obtained in \cite{Bardos}. Compare to the previous study for the case $\theta\kappa + \kappa + 3\theta - 1> 0$, ours have the advantage that for $\theta\kappa + \kappa + 3\theta - 1 \ge 0$, we work on the weaker regularity assumptions, and the effective technique is applied to extend the proof. The key idea comes from the better estimations on regularization operators that will be described later.

The rest of the paper is organized as follows. Next section \ref{sec:main} is devoted to some notations and definitions about the renormalization property and entropy conservation laws, and our main result of this paper is also stated therein. We then introduce in Section~\ref{sec:reg} some regularization operators and properties are also presented for later use. Finally, the last section gives a brief proof of the renormalization property and entropy conservation laws for Diperna-Lions weak solution to the Vlasov-Maxwell equations.
 
\section{Main result}
\label{sec:main}
Let us start this section by introducing some general notation and definitions concerning to the problem which will be used in the whole paper. In the sequel, we denote by $\mathcal{D}(\mathbb{R}^n)$, with $n \ge 1$, the space of infinitely differentiable functions with compact support and by $\mathcal{D}'(\mathbb{R}^n)$ the space of distribution. 

On the other hand, for $\theta \in (0,1)$, $1 \le p \le \infty$, the generalized fractional order Sobolev spaces $W^{\theta,p}(\mathbb{R}^n)$ is defined for any function $f$ belonging to $W^{\theta,p}(\mathbb{R}^n)$ if and only if the following Gagliardo-type norm is finite:
\begin{align*}
\|f\|_{W^{\theta,p}(\mathbb{R}^n)}:= \left(\int_{\mathbb{R}^n}|f(x)|^pdx\right)^{1/p} + \left(\int_{\mathbb{R}^n}\int_{\mathbb{R}^n} \frac{|f(x)-f(y)|^p}{|x-y|^{n+\theta p}}dxdy\right)^{1/p}< +\infty,
\end{align*}
in the case $1 \le p < \infty$ and
\begin{align*}
\|f\|_{W^{\theta,\infty}(\mathbb{R}^n)}:= \max\left\{\|f\|_{L^{\infty}(\mathbb{R}^n)}, \sup_{x\neq y \in \mathbb{R}^n} \frac{|f(x)-f(y)|}{|x-y|^{\theta}}\right\} < +\infty,
\end{align*}
for $p = \infty$.
Here and subsequently, $\mathcal{L}^1(\mathbb{R}^6)$ denotes the set of non-negative almost everywhere function $f$ such that
\begin{align*}
 \|f\|_{\mathcal{L}^1(\mathbb{R}^6)}: = \int_{\mathbb{R}^6} f(x,\varsigma) \sqrt{1+|\varsigma|^2} \, dx \, d\varsigma < +\infty.
\end{align*}
In addition, the notation $\mathcal{S}$ stands for the set of non-decreasing function $\mathcal{G} \in \mathcal{C}^1(\mathbb{R}^+;\mathbb{R}^+)$ such that $\lim_{t\to +\infty} \frac{\mathcal{G}(t)}{t} = +\infty.$

The weak solution of a coupled set of relativistic Vlasov-Maxwell equations involves the distribution function $u$ (describes plasma components), electric and magnetic fields $E, B$ (self-consistenly modified by particles). Here, we say that $(u,E,B)$ is a weak solution to relativistic Vlasov-Maxwell equations~\eqref{eq:sys1} if $(u,E,B)$ satisfies the following weak formulation
$$
 \int_0^T dt \int_{\mathbb{R}^3}dx \int_{\mathbb{R}^3}  u(\partial_t \varphi + v\cdot \nabla_x \varphi + \mathcal{F} \cdot \nabla_{\varsigma}\varphi) \, d\varsigma = 0,
$$
for all $\varphi \in \mathcal{D}((0,T)\times \mathbb{R}^6)$.
The existence result of a global in time weak solution to the relativistic Vlasov-Maxwell equations proposed by DiPerna-Lions is restated in the following theorem, see~\cite{Diperna} to which we refer the interested
readers.
\begin{theorem}\label{theo:1}
Let $u_0 \in \mathcal{L}^1 \cap L^{\infty}(\mathbb{R}^6)$ and $E_0$, $B_0 \in L^2(\mathbb{R}^3)$ be initial conditions with satisfy the constraints
$$\normalfont{\mbox{div}}\, B_0 = 0, \quad \normalfont{\mbox{div}}\, E_0 = \int_{\mathbb{R}^3}u_0 \,d\varsigma, \quad \mbox{ in }\ \mathcal{D}'(\mathbb{R}^3).$$
Then there exists a global in time weak solution of the relativistic Vlasov-Maxwell system, i.e., there exist functions
\begin{align*}
\begin{split}
u \in L^{\infty}(\mathbb{R}^+;\mathcal{L}^1\cap L^{\infty}(\mathbb{R}^6)), \ &  \ E, B \in L^{\infty}(\mathbb{R}^+;L^2(\mathbb{R}^3)), \\
 \mbox{ and } \ \ \rho, j \in L^{\infty}&(\mathbb{R}^+;L^{4/3}(\mathbb{R}^3)),
 \end{split}
\end{align*}
such that $(u,E,B)$ satisfy~\eqref{eq:sys1}-\eqref{eq:sys2b} in the sense of distributions, where $\rho$ and $j$ are defined in~\eqref{eq:sys3}. 
\end{theorem}

Let $(u, E, B)$ be a weak solution to the relativistic Vlasov-Maxwell system~\eqref{eq:sys1}-\eqref{eq:sys3}, as in Theorem~\ref{theo:1}. Then for any smooth function $\mathcal{G} \in \mathcal{C}^1(\mathbb{R}^+; \mathbb{R}^+)$, we say that $(u, E, B)$ satisfies the renormalization property if 
\begin{align}\label{eq:renorl_property}
\partial_t(\mathcal{G}(u)) + \nabla_x \cdot (v \mathcal{G}(u)) + \nabla_{\varsigma} \cdot (\mathcal{F}\mathcal{G}(u)) = 0, \quad \mbox{ in } \ \mathcal{D}'((0,T)\times \mathbb{R}^6)
\end{align}
in the sense of distribution, that means, 
\begin{align*}
\int_0^T dt \int_{\mathbb{R}^3}dx \int_{\mathbb{R}^3}  \mathcal{G}(u) \left( \partial_t \varphi + v \cdot \nabla_x \varphi + \mathcal{F} \cdot \nabla_{\varsigma} \varphi\right) \, d\varsigma  = 0, 
\end{align*}
for all $\varphi \in \mathcal{D}'((0,T)\times \mathbb{R}^6)$. Otherwise, solution $(u, E, B)$ is said to satisfy the local in space entropy conservation law, if
\begin{align}\label{eq:local_law1}
\partial_t\left(\int_{\mathbb{R}^3} \mathcal{G}(u)d\varsigma\right) + \nabla_x \cdot \left(\int_{\mathbb{R}^3} v\mathcal{G}(u)d\varsigma\right) = 0, \quad \mbox{in } \ \mathcal{D}'((0,T)\times\mathbb{R}^3),
\end{align}
and the local in momentum entropy conservation law, if
\begin{align}\label{eq:local_law2}
\partial_t\left(\int_{\mathbb{R}^3} \mathcal{G}(u)dx\right) + \nabla_{\varsigma} \cdot \left(\int_{\mathbb{R}^3} \mathcal{F}\mathcal{G}(u)dx\right) = 0, \quad \mbox{in } \ \mathcal{D}'((0,T)\times\mathbb{R}^3),
\end{align}
in the sense of distribution. In this way, we can state that $(u, E, B)$ satisfies the global entropy conservation law, if we have
\begin{align}\label{eq:global_law}
\int_{\mathbb{R}^6} \mathcal{G}(u(t_2,x,\varsigma)) \, d\varsigma \, dx = \int_{\mathbb{R}^6} \mathcal{G}(u(t_1,x,\varsigma))\, d\varsigma \, dx, \quad \mbox{ for } 0< t_1\le t_2 < T.
\end{align}

Let us state our main result about the renormalized property and entropy conservation laws for the global weak solution of relativistic Vlasov-Maxwell equations. Related to the present note, it emphasizes that the regularity assumptions on the weak solution in our work are weaker than in the paper of C. Bardos {\it et al.}~\cite{Bardos}, our improved results thus are more general. In particular, we prove that the renormalization property and entropy conservation laws for the global weak in time solution to the relativistic Vlasov-Maxwell's system~\eqref{eq:sys1}-\eqref{eq:sys3} even hold for the endpoint case $\theta\kappa + \kappa + 3\theta - 1 = 0$, as described in the following theorem.
\begin{theorem}\label{theo:main}
Let $(u, E, B)$ be a weak solution of the relativistic Vlasov-Maxwell system~\eqref{eq:sys1}-\eqref{eq:sys3} given by Theorem~\ref{theo:1}. Assume moreover that this weak solution satisfies the additional regularity assumptions
\begin{align}\label{cond:fEB}
u \in L^{\infty}(0,T;W^{\theta,p}(\mathbb{R}^6)) \quad \mbox{and} \quad E, B \in L^{\infty}(0,T;W^{\kappa,q}(\mathbb{R}^3)),
\end{align}
where $\theta, \kappa \in (0,1)$ such that $\theta \kappa + \kappa + 3 \theta - 1 \ge 0,$
and $p, q \in \mathbb{N}^*$ such that
\begin{align}\label{cond:pq}
\frac{1}{p} + \frac{1}{q}  = \frac{1}{r} \le 1 \ \mbox{if} \ 1 \le p,\,q<\infty, \ \mbox{and} \ 1 \le r <\infty \ \mbox{is arbitrary if} \ p = q = \infty.
\end{align}
Then for any entropy function $\mathcal{G} \in \mathcal{C}^1(\mathbb{R}^+; \mathbb{R}^+)$, the global weak solution $(u, E, B)$ satisfies the renormalization property~\eqref{eq:renorl_property}. Moreover, if $\mathcal{G} \in \mathcal{S}$ and the mapping $t \mapsto u(t,\cdot,\cdot)$ is uniformly integrable in $\mathbb{R}^6$, for almost everywhere $t \in [0,T]$, then the local entropy conservation laws~\eqref{eq:local_law1}-\eqref{eq:local_law2} and the global entropy conservation law~\eqref{eq:global_law} hold.
\end{theorem}

\section{Regularization operators}\label{sec:reg}
In this section, let us mention the important consequences of this work, that leading to the proof of our main result. It is devoted to study the standard regularization operators and their properties, that gives us the idea to prove main result in this paper. We will now show their descriptions and prove some preparatory lemmas that are necessary for later use.

Let $\phi \in \mathcal{D}(\mathbb{R}^+;\mathbb{R}^+)$ be a smooth non negative function such that
\begin{align}\label{eq:varrho}
 {\mbox{supp}(\phi) \subset \left[1,2\right]}, \quad \int_{\mathbb{R}} \phi(\nu)d\nu = 1.
\end{align}
For every $\varepsilon>0$ and $n \in \mathbb{N}^*$, the radially-symmetric compactly-supported Friedrichs mollifier is given by
\begin{align}\label{eq:varrho_eps}
\phi_{\varepsilon}: \ \mathbb{R}^n \rightarrow \mathbb{R}^+, \quad x \mapsto \phi_{\varepsilon}(x) = {\varepsilon^{-n}} \phi\left(\varepsilon^{-1}|x|\right).
\end{align}
Let $\gamma, \varepsilon, \sigma$ be positive numbers and for  any distribution $f \in \mathcal{D}'(\mathbb{R}^n)$, $g \in \mathcal{D}'(\mathbb{R}^+\times\mathbb{R}^n)$ and $h \in \mathcal{D}'(\mathbb{R}^+\times\mathbb{R}^n\times \mathbb{R}^n)$ , we define their $\mathcal{C}^{\infty}$-regularization by
\begin{align*}
f^{\sigma}(x) :=  &\phi_{\sigma}(x) * f(x), \ \ g^{\varepsilon,\sigma}(t,x) := \phi_{\varepsilon}(t) *_t \phi_{\sigma}(x) *_x g(t,x), \ \mbox{ and} \\
& h^{\gamma,\varepsilon,\sigma}(t,x,\varsigma)  := \phi_{\gamma}(t) *_t \phi_{\varepsilon}(x) *_x\phi_{\sigma}(\varsigma) *_{\varsigma} h(t,x,\varsigma),
\end{align*}
where the operator $*$ denotes the standard convolution product. 

Before giving the proof of the main result, let us provide some preliminary lemmas. The following lemma establishes two basic estimations for the relativistic velocity $v$ in~\eqref{eq:sys4}.
\begin{lemma}\label{lem:1b} 
Let $\sigma>0$ and $v$ be the relativistic velocity given by~\eqref{eq:sys4}. Then for all $\varsigma, w \in \mathbb{R}^3$, we have the following estimations
\begin{align}\label{eq:est_v}
& |v(\varsigma-w) - v(\varsigma)| \le 2 |\upsilon|,\\ \label{eq:est_v_del}
& \left|v(\varsigma) - v^{\sigma}(\varsigma)\right| \le 4\sigma.
\end{align}
\end{lemma}
\begin{proof}
By a simple computation, we firstly get that
\begin{align}\label{eq:est1}
|\nabla_{\varsigma} v| = \left| \frac{I_3}{\sqrt{1+|\varsigma|^2}} - \frac{\varsigma \otimes \varsigma}{\sqrt{(1+|\varsigma|^2)^3}}\right|  \le 2,
\end{align}
where $I_3$ denotes the identity matrix of size 3. Combining inequality~\eqref{eq:est1} and the fundamental theorem of calculus, we obtain the first basic estimate~\eqref{eq:est_v},
\begin{align*}
|v(\varsigma-w) - v(\varsigma)| \le |\upsilon|\int_0^1 |\nabla v(\varsigma - s \upsilon)|ds \le 2 |\upsilon|.
\end{align*}
By using this estimation, we obtain the second basic estimate~\eqref{eq:est_v_del} as follows
\begin{align*}
\left|v(\varsigma) - v^{\sigma}(\varsigma)\right| = \left|\int_{\mathbb{R}^3} \phi_{\sigma}(\upsilon)(v(\varsigma)-v(\varsigma-w)) d\upsilon \right| \le 2\left|\int_{\mathbb{R}^3} \phi_{\sigma}(\upsilon) |\upsilon| d\upsilon \right| \le 4\sigma.
\end{align*}
\end{proof}

Some well-known properties for $\mathcal{C}^{\infty}$-regularization will be presented in the next lemma, and for the proofs of $(ii)$ and $(iii)$, it refers the reader to some fine papers found in~\cite[Proposition 4.2]{Cormick2013} or in~\cite{Alinhac}.

\begin{lemma}\label{lem:1} 
\begin{enumerate}
\item[(i)] Let $f \in \mathcal{D}'(\mathbb{R}^n)$, $g \in \mathcal{D}(\mathbb{R}^n)$ and $\varepsilon>0$, there holds $\langle f^{\varepsilon}, g \rangle = \langle f, g^{\varepsilon} \rangle,$ where $\langle \cdot, \cdot \rangle$ denotes the dual bracket between spaces $\mathcal{D}'$ and $\mathcal{D}$.
\item[(ii)] Let $\varepsilon>0$, $\theta \in (0,1)$, $1 \le p \le \infty$ and $f \in L^1 \cap L^{\infty} \cap W^{\theta, p}(\mathbb{R}^n)$, there holds
\begin{align*}
\|f^{\varepsilon}\|_{L^p(\mathbb{R}^n)} \le \|f\|_{L^p(\mathbb{R}^n)} \ \mbox{ and } \ \|f^{\varepsilon}\|_{W^{\theta,p}(\mathbb{R}^n)} \le \|f\|_{W^{\theta,p}(\mathbb{R}^n)}.
\end{align*}
\item[iii)] Let $\theta \in (0,1)$ and $1 \le p \le \infty$. Then for any $f \in W^{\theta,p}(\mathbb{R}^n)$, there exists a constant $C$ such that
\begin{align}\label{eq:f_w}
\|f(\cdot - w) - f(\cdot)\|_{L^p(\mathbb{R}^n)} \le C |\upsilon|^{\theta} \|f\|_{W^{\theta,p}(\mathbb{R}^n)}, \ \forall  w \in \mathbb{R}^n.
\end{align}
\end{enumerate}
\end{lemma}
We now derive some crucial properties via the following technical lemmas, that will play a critical role into deriving our main results (still hold for the endpoint case) in the rest of the paper.
\begin{lemma}\label{lem:1a} $ $
Let $\varepsilon>0$, $\theta \in (0,1)$ and $1 \le p \le \infty$. Then for any function $f$ belongs to ${L}^1\cap L^{\infty} \cap W^{\theta, p}(\mathbb{R}^n)$, there exists a constant $C$  such that
\begin{align}\label{eq:reg_f}
 \|f^{\varepsilon}-f\|_{L^p(\mathbb{R}^n)} \le C \varepsilon^{\theta} \left(\int_{\mathbb{R}^n}\int_{\mathbb{R}^n} \mathbbm{1}_{\varepsilon\le |x-y|\le 2\varepsilon} \frac{|f(x)-f(y)|^p}{|x-y|^{n+\theta p}} dx dy\right)^{\frac{1}{p}}, 
\end{align}
and
\begin{align}\label{eq:reg_gradf}
 \|\nabla f^{\varepsilon}\|_{L^p(\mathbb{R}^n)} \le C\varepsilon^{\theta-1} \left(\int_{\mathbb{R}^n}\int_{\mathbb{R}^n} \mathbbm{1}_{\varepsilon\le |x-y|\le 2\varepsilon} \frac{|f(x)-f(y)|^p}{|x-y|^{n+\theta p}} dx dy\right)^{\frac{1}{p}}.
\end{align}
\end{lemma}
\begin{proof}
For any $y \in \mathbb{R}^n$, by the definition of $\phi_{\varepsilon}$ in \eqref{eq:varrho_eps} and using H{\"o}lder inequality, we obtain the following estimation
\begin{align*}
|f^{\varepsilon}(y) - f(y)| &= \left|\int_{\mathbb{R}^n} \phi_{\varepsilon}(y-x) (f(x) - f(y)) dx\right| \\
& \le C\varepsilon^{-n}\int_{\mathbb{R}^n} \mathbbm{1}_{\varepsilon\le |x-y|\le 2\varepsilon}|f(x) - f(y)| dx \\
& \le C\varepsilon^{-\frac{n}{p}} \left(\int_{\mathbb{R}^n} \mathbbm{1}_{\varepsilon\le |x-y|\le 2\varepsilon}|f(x) - f(y)|^p dx\right)^{\frac{1}{p}},
\end{align*}
where the constant $C$ depends only on the function $\phi$ given in~\eqref{eq:varrho}. It follows that
\begin{align*}
\|f^{\varepsilon}-f\|_{L^p(\mathbb{R}^n)}^p \le C\varepsilon^{-n}  \int_{\mathbb{R}^n} \int_{\mathbb{R}^n} \mathbbm{1}_{\varepsilon\le |x-y|\le 2\varepsilon}|f(x) - f(y)|^p dx dy.
\end{align*}
Otherwise, by multiplying two sides of this inequality by $\varepsilon^{-\theta p}$, it gives
\begin{align*}
\varepsilon^{-\theta p}\|f^{\varepsilon}-f\|_{L^p(\mathbb{R}^n)}^p & \le \frac{C}{\varepsilon^{n+\theta p}} \int_{\mathbb{R}^n}\int_{\mathbb{R}^n} \mathbbm{1}_{\varepsilon\le |x-y|\le 2\varepsilon}|f(x) - f(y)|^p dx dy \\
& \le C \int_{\mathbb{R}^n}\int_{\mathbb{R}^n} \mathbbm{1}_{\varepsilon\le |x-y|\le 2\varepsilon} \frac{|f(x)-f(y)|^p}{|x-y|^{n+\theta p}} dx dy,
\end{align*}
which deduces the first inequality~\eqref{eq:reg_f}. In order to obtain the second estimation, it will be necessary to remark that
\begin{align*}
|\nabla f^{\varepsilon}(y)| & = \left|\int_{\mathbb{R}^n} \nabla\phi_{\varepsilon}(y-x) (f(x) - f(y)) dx\right|\\
& \le \frac{C}{\varepsilon^{n+1}}\int_{\mathbb{R}^n} \mathbbm{1}_{\varepsilon\le |x-y|\le 2\varepsilon}|f(x) - f(y)| dx.
\end{align*}
By the same proof of~\eqref{eq:reg_f}, we obtain~\eqref{eq:reg_gradf} the desired result.
\end{proof}

\begin{remark}\label{remk}
For all $f \in W^{\theta,p}(\mathbb{R}^n)$, one can see that
\begin{align*}
 \int_{\mathbb{R}^n} \int_{\mathbb{R}^n} \mathbbm{1}_{\varepsilon\le |x-y|\le 2\varepsilon} \frac{|f(x) - f(y)|^p}{|x-y|^{n + \theta p}} dx dy \le \|f\|_{W^{\theta,p}(\mathbb{R}^n)}.
\end{align*}
Therefore, as the consequences of Lemma~\ref{lem:1a}, one also obtains
\begin{align*}
 \|f^{\varepsilon}-f\|_{L^p(\mathbb{R}^n)} \le C \varepsilon^{\theta} \|f\|_{W^{\theta,p}(\mathbb{R}^n)} \ \mbox{ and } \ \|\nabla f^{\varepsilon}\|_{L^p(\mathbb{R}^n)} \le C\varepsilon^{\theta-1}  \|f\|_{W^{\theta,p}(\mathbb{R}^n)}.
\end{align*}
\end{remark}
For every function $f \in W^{\theta,p}(\mathbb{R}^n\times\mathbb{R}^n)$, we define a function $\Theta_f$ as
\begin{align}\label{eq:Theta}
\Theta_{f}(\varepsilon) := \left(\int_{\mathbb{R}^n}\int_{\mathbb{R}^n}\int_{\mathbb{R}^n} \mathbbm{1}_{\varepsilon\le |x-y|\le 2\varepsilon} \frac{|f(x,\varsigma)-f(y,\varsigma)|^p}{|x-y|^{n+\theta p}} dx dy d\varsigma\right)^{\frac{1}{p}},
\end{align}
the following Lemma is then stated and proved to give us a very important property related to this function.

\begin{corollary}\label{cor:1} Let $\theta\in (0,1)$, $1 \le p \le \infty$ and the function $f \in W^{\theta,p}(\mathbb{R}^n\times\mathbb{R}^n)$. Then for any $\varepsilon,\,\sigma>0$, there exists a constant $C$ such that
\begin{align}\label{eq:cor_a}
 \|\nabla_x f^{\varepsilon}(x,\varsigma-w) - & \nabla_x f^{\varepsilon}(x,\varsigma)\|_{L^p(\mathbb{R}^n\times\mathbb{R}^n)}  \le C \varepsilon^{\theta-1} |\upsilon|^{\theta}   \Theta_{f}(\varepsilon),
\end{align}
with $w \in \mathbb{R}^n$, and
\begin{align}\label{eq:cor_b}
\| (\nabla_x f^{\varepsilon})^{\sigma} - \nabla_x f^{\varepsilon}\|_{L^p(\mathbb{R}^n\times\mathbb{R}^n)} \le C \varepsilon^{\theta-1} \sigma^{\theta}   \Theta_{f}(\varepsilon),
\end{align}
where the function $\Theta_f$ is defined by~\eqref{eq:Theta}.
\end{corollary}
\begin{proof}
From Lemma~\ref{lem:1}, there exists a constant $C$ such that
\begin{align*}
 \|\nabla_x f^{\varepsilon}(x,\varsigma-w) - & \nabla_x f^{\varepsilon}(x,\varsigma)\|_{L^p(\mathbb{R}^n\times\mathbb{R}^n)}  \le C  |\upsilon|^{\theta} \|\nabla_x f^{\varepsilon}\|_{L^p(\mathbb{R}^n;W^{\theta,p}(\mathbb{R}^n))}.
\end{align*}
The inequality~\eqref{eq:reg_gradf} in Lemma~\ref{lem:1a} is then applied to get
\begin{align*}
 \|\nabla_x f^{\varepsilon}(x,\varsigma-w) - & \nabla_x f^{\varepsilon}(x,\varsigma)\|_{L^p(\mathbb{R}^n\times\mathbb{R}^n)}  \le C \varepsilon^{\theta-1} |\upsilon|^{\theta}   \Theta_{f}(\varepsilon).
\end{align*}
To deal with the second estimation~\eqref{eq:cor_b}, by what obtained in Lemma~\ref{lem:1a} and Remark~\ref{remk}, there exists a constant $C$ such that
\begin{align*}
 \| (\nabla_x f^{\varepsilon})^{\sigma} - \nabla_x f^{\varepsilon}\|_{L^p(\mathbb{R}^n\times\mathbb{R}^n)} \le C \sigma^{\theta} \|\nabla_x f^{\varepsilon}\|_{L^p(\mathbb{R}^n;W^{\theta,p}(\mathbb{R}^n))}.
\end{align*}
Repeated application of the inequality~\eqref{eq:reg_gradf} in Lemma~\ref{lem:1a} enables us to write
\begin{align*}
 \| (\nabla_x f^{\varepsilon})^{\sigma} - \nabla_x f^{\varepsilon}\|_{L^p(\mathbb{R}^n\times\mathbb{R}^n)} \le C \varepsilon^{\theta-1} \sigma^{\theta} \Theta_{f}(\varepsilon),
\end{align*}
and the proof is complete.
\end{proof}

\begin{lemma}\label{lem:omega}
Let $\theta \in (0,1)$, $1 \le p \le \infty$ and the function $f \in L^{1}(0,T;W^{\theta,p}(\mathbb{R}^n\times\mathbb{R}^n))$. Then $\omega_f(\varepsilon,\sigma)$ defined by 
\begin{align}\label{def:omega}
\omega_f(\varepsilon,\sigma) := \int_0^T (\Theta_{f(t)}(\varepsilon) + \Theta_{f(t)}(\sigma))dt,
\end{align}
vanishes as $\varepsilon$ and $\sigma$ tend to 0.
\end{lemma}
\begin{proof}
By the definition of $\Theta_f$ in~\eqref{eq:Theta} and Remark~\ref{remk}, we have
\begin{align*}
\Theta_{f(t)}(\varepsilon) \le \|f(t,\cdot,\cdot)\|_{W^{\theta,p}(\mathbb{R}^n\times\mathbb{R}^n)} < \infty, \quad \forall t \in [0,T].
\end{align*}
Apply Lebesgue dominated convergence theorem, it is clear that $\Theta_{f(t)}(\varepsilon)$ tends to 0 as passing $\varepsilon$ goes to 0 for all $t \in [0,T]$. The same conclusion is obtained for $\Theta_{f(t)}(\sigma)$, and this guarantees that $\omega_f(\varepsilon,\sigma)$ given by~\eqref{def:omega} vanishes as $(\varepsilon, \sigma)$ goes to 0. 
\end{proof}

\section{Proof of Theorem~\ref{theo:main}}\label{sec:proof}
In this section, we consider a global in time weak solution $(u,E,B)$ of Vlasov-Maxwell equations. The weak formulation for the Vlasov equation~\eqref{eq:sys1} reads
$$
 \int_0^T dt \int_{\mathbb{R}^3}dx \int_{\mathbb{R}^3}  u(\partial_t \varphi + v\cdot \nabla_x \varphi + \mathcal{F} \cdot \nabla_{\varsigma}\varphi)  \, d\varsigma = 0,
$$
for all $\varphi \in \mathcal{D}((0,T)\times \mathbb{R}^6)$. Let us choose a test function $\varphi$ as follows
$$
\varphi =  (\mathcal{G}'(u^{\gamma,\varepsilon,\sigma})\psi)^{\gamma,\varepsilon,\sigma} \in \mathcal{D}((0,T)\times \mathbb{R}^6),
$$
where $\psi \in \mathcal{D}((0,T)\times \mathbb{R}^6)$ and $\mathcal{G} \in \mathcal{C}^1(\mathbb{R}^+;\mathbb{R}^+)$. Integrating by parts this weak formulation yields that for all $\psi \in \mathcal{D}((0,T)\times \mathbb{R}^6)$, there holds
\begin{multline} \label{eq:explain1}
\int_0^T dt \int_{\mathbb{R}^3}dx \int_{\mathbb{R}^3} \, d\varsigma  \mathcal{G}(u^{\gamma,\varepsilon,\sigma})\left(\partial_t \psi + v^{\sigma}\cdot \nabla_x \psi + \mathcal{F}^{\gamma,\varepsilon,\sigma}\cdot \nabla_{\varsigma} \psi\right) \\ + \psi \mathcal{G}'(u^{\gamma,\varepsilon,\sigma}) \left[\nabla_x \cdot \left((vu)^{\gamma,\varphi,\sigma}-v^{\sigma}u^{\gamma,\varepsilon,\sigma}\right) + \nabla_{\varsigma} \cdot\left((\mathcal{F}u)^{\gamma,\varepsilon,\sigma}-\mathcal{F}^{\gamma,\varepsilon,\sigma}u^{\gamma,\varepsilon,\sigma}\right)\right]  = 0.
\end{multline}

Following the renormalization property of solution $(u,E,B)$, it is sufficient to show that the second term in the left hand side of~\eqref{eq:explain1} vanishes as $(\gamma,\varepsilon,\sigma)$ tends to 0, for all $\psi \in \mathcal{D}((0,T)\times \mathbb{R}^6)$. To do so, we firstly establish some commutator estimations which are presented in the next lemma. For simplicity, the problem is considered with $\theta \in (0,1)$ and $1 \le p, r \le \infty$, with $n = 3$ or $n = 6$ and $s = 1$ or $s = \infty$, we will use the following notations in the remain part of our paper, 
\begin{align*}
 \mathcal{L}^{s,p}_n := L^s(0,T;L^p(\mathbb{R}^n)), & \ \ \mathcal{L}^{s,p,r} := L^s(0,T;L^p(\mathbb{R}^3;L^r(\mathbb{R}^3))),\\
 \mathcal{L}^{s}\mathcal{W}^{\theta,p}_n := L^s(0,T;W^{\theta,p}(\mathbb{R}^n)), & \ \  \mathcal{L}^{s,p}\mathcal{W}^{\theta,p} := L^s(0,T;L^p(\mathbb{R}^3;W^{\theta,p}(\mathbb{R}^3))).
\end{align*}

Following the idea of the proofs in \cite{Bardos}, it is worth emphasizing that in our proofs of Lemma \ref{lem:2} and Theorem \ref{theo:main} below, we improve the earlier version (with the endpoint case included) based on technical lemmas \ref{lem:1a}, \ref{cor:1} and \ref{lem:omega} in previous section.

\begin{lemma}\label{lem:2}
Let $(u,E,B)$ be a weak solution of the relativistic Vlasov-Maxwell system \eqref{eq:sys1}-\eqref{eq:sys3} given by Theorem~\ref{theo:1}, satisfying the regularity assumptions~\eqref{cond:fEB} of Theorem~\ref{theo:main}, with $\theta, \kappa \in (0,1)$ and $p, q, r$ satisfy relations \eqref{cond:pq}. Then for any positive numbers $\gamma, \varepsilon, \sigma >0$, there exists a constant $C>0$ depending only on the smooth function $\phi$ given by~\eqref{eq:varrho} such that
\begin{align}\label{eq:Cfs}
\left\|\nabla_x \cdot \left((vu)^{\gamma,\varepsilon,\sigma}-v^{\sigma}u^{\gamma,\varepsilon,\sigma}\right)\right\|_{ \mathcal{L}^{1,p}_6} \le C \varepsilon^{\theta-1} \sigma^{\theta+1} \omega_u(\varepsilon,\sigma).
\end{align}
Moreover, there exists a constant $C_{\mathcal{F}}>0$ depending on $\phi$, $\|u\|_{\mathcal{L}^{1}\mathcal{W}^{\theta,p}_6}$, $\|E\|_{\mathcal{L}^{\infty}\mathcal{W}^{\kappa,q}_3}$ and $\|B\|_{\mathcal{L}^{\infty}\mathcal{W}^{\kappa,q}_3}$ such that
\begin{align}\label{eq:Cfl}
\left\|\nabla_\varsigma \cdot \left((\mathcal{F}u)^{\gamma,\varepsilon,\sigma}-\mathcal{F}^{\gamma,\varepsilon,\sigma}u^{\gamma,\varepsilon,\sigma}\right)\right\|_{\mathcal{L}^{1,p,r}}  \le C_{\mathcal{F}} \left(\varepsilon^{\theta+\kappa} \sigma^{\theta-1}\omega_u(\varepsilon,\sigma)  + \sigma^{\theta}\right),
\end{align}
where $\mathcal{F}: = E + v \times B$ is the Lorentz force field and the function $\omega_u$ is given by~\eqref{def:omega}.
\end{lemma}
\begin{proof}
We first consider the commutator estimate~\eqref{eq:Cfs} for the free streaming term. It is easy to check that
\begin{align}\label{eq:est6}
(vu)^{\gamma,\varepsilon,\sigma} - v^{\sigma} u^{\gamma,\varepsilon,\sigma} = \mathcal{K}_{\sigma}(v,u^{\gamma,\varepsilon})  - (u^{\gamma,\varepsilon,\sigma}-u^{\gamma,\varepsilon})(v-v^{\sigma}),
\end{align}
where $\mathcal{K}_{\sigma}$ is defined by
\begin{align}\label{eq:est5}
\mathcal{K}_{\sigma}(v,g)(t,x,\varsigma) =  \int_{\mathbb{R}^3}   \phi_{\sigma}(\upsilon) \left(v(\varsigma-w)-v(\varsigma)\right)\left(g(t,x,\varsigma-w)-g(t,x,\varsigma)\right) d\upsilon.
\end{align}
Passing to the limit $\gamma \to 0$ on the right hand side of~\eqref{eq:est6} which can be justified by the Lebesgue dominated convergence theorem and regularity assumptions~\eqref{cond:fEB}, we thus get that
\begin{align*}
\left\|\nabla_x \cdot \left((vu)^{\gamma,\varepsilon,\sigma}-v^{\sigma}u^{\gamma,\varepsilon,\sigma}\right)\right\|_{ \mathcal{L}^{1,p}_6} \le \mathcal{A}_1 + \mathcal{A}_2, 
\end{align*}
where $ \mathcal{A}_1 = \|\nabla_x \cdot \mathcal{K}_{\sigma}(v,u^{\varepsilon})\|_{ \mathcal{L}^{1,p}_6}$  and  $\mathcal{A}_2 = \|\nabla_x \cdot ((u^{\varepsilon,\sigma}-u^{\varepsilon})  (v-v^{\sigma}))  \|_{ \mathcal{L}^{1,p}_6}$.
By the definition of $\mathcal{K}_{\sigma}$ in~\eqref{eq:est5}, and in the use of~\eqref{eq:est_v} in Lemma~\ref{lem:1b} and~\eqref{eq:cor_a} in Corollary~\ref{cor:1},  we obtain  that
\begin{align*} 
 \mathcal{A}_1 & \le C \varepsilon^{\theta-1} \int_0^T  \int_{\mathbb{R}^3}  \phi_{\sigma}(\upsilon) |\upsilon|^{\theta+1} \Theta_{u(t)}(\varepsilon) d\upsilon dt \le C \varepsilon^{\theta-1}\sigma^{\theta+1} \omega_u(\varepsilon,\sigma),
\end{align*}
where $\omega_u$ given in~\eqref{def:omega}. Additionally, from~\eqref{eq:est_v_del} and~\eqref{eq:cor_b}, there holds
\begin{align*}
\mathcal{A}_2 \le  \| |v - v^{\sigma}| ( \nabla_x u^{\varepsilon,\sigma} - \nabla_x u^{\varepsilon})\|_{ \mathcal{L}^{1,p}_6} \le C \varepsilon^{\theta-1} \sigma^{\theta+1} \omega_u(\varepsilon,\sigma).
\end{align*}
From what have already been proved, we obtain commutator estimate \eqref{eq:Cfs}. 

It remains to prove the estimate in~\eqref{eq:Cfl}. To establish this commutator estimate for the Lorentz force term, it is possible for us to make the following decomposition as follows
\begin{align}\nonumber
(\mathcal{F}u)^{\gamma,\varepsilon,\sigma} - \mathcal{F}^{\gamma,\varepsilon,\sigma}u^{\gamma,\varepsilon,\sigma} &= \mathcal{K}_{\gamma,\varepsilon}(E,u^{\sigma}) - (E - E^{\gamma,\varepsilon})(u^{\sigma}-(u^{\sigma})^{\gamma,\varepsilon}) \\ \label{eq:TE_TB} & \qquad \qquad \qquad \qquad + (v\times Bu)^{\varepsilon,\sigma} - v^{\sigma}\times B^{\varepsilon} u^{\varepsilon,\sigma},
\end{align}
where $\mathcal{K}_{\gamma,\varepsilon}$ is given by
\begin{multline*}
\mathcal{K}_{\gamma,\varepsilon}(E,g)(t,x,\varsigma) =  \int_{\mathbb{R}} d\nu \int_{\mathbb{R}^3} dy \  \phi_{\gamma} (\nu) \phi_{\varepsilon}(y)  \\  \quad \cdot\left(E(t-\nu,x-y,\varsigma)-E(t,x,\varsigma)\right)\left(g(t-\nu,x-y,\varsigma)-g(t,x,\varsigma)\right).
\end{multline*}
For the sake of simplicity, in this work we will denote 
\begin{align}\label{eq:TB}
 T_E := \mathcal{K}_{\gamma,\varepsilon}(E,u^{\sigma}) - (E - E^{\gamma,\varepsilon})(u^{\sigma}-(u^{\sigma})^{\gamma,\varepsilon}),\ \ T_B := (v\times Bu)^{\varepsilon,\sigma} - v^{\sigma}\times B^{\varepsilon} u^{\varepsilon,\sigma},
\end{align}
and make the effective use of the Lebesgue dominated convergence theorem together with regularity assumptions~\eqref{cond:fEB}, passing to the limit $\gamma \to 0$ in $T_E$, yields that
\begin{align} \label{eq:TE_Leb2}
\|\nabla_{\varsigma} \cdot  T_E\|_{\mathcal{L}^{1,p,r}}  \le  \mathcal{A}_3 + \mathcal{A}_4, 
\end{align} 
where $\mathcal{A}_3 := \|\nabla_{\varsigma} \cdot \mathcal{K}_{\varepsilon}(E,u^{\sigma})\|_{\mathcal{L}^{1,p,r}}$  and $\mathcal{A}_4 := \|\nabla_{\varsigma} \cdot ((E - E^{\varepsilon})(u^{\sigma}-(u^{\sigma})^{\varepsilon}))\|_{\mathcal{L}^{1,p,r}}.$ By H{\"o}lder inequality, there holds
\begin{align}\nonumber
 \mathcal{A}_3 & \le \int_{\mathbb{R}^3} \phi_{\varepsilon}(y) \|(E(t,x-y)- E(t,x))  \cdot (\nabla_{\varsigma}u^{\sigma}(t,x-y,\varsigma)-\nabla_{\varsigma}u^{\sigma}(t,x,\varsigma))\|_{\mathcal{L}^{1,p,r}} dy\\ \nonumber
&\le \int_{\mathbb{R}^3} \phi_{\varepsilon}(y) \|E(t,x-y)- E(t,x)\|_{\mathcal{L}^{\infty,q}_3} \|\nabla_{\varsigma}u^{\sigma}(t,x-y,\varsigma)-\nabla_{\varsigma}u^{\sigma}(t,x,\varsigma)\|_{ \mathcal{L}^{1,p}_6} dy.
\end{align} 
Applying the estimate~\eqref{eq:f_w} in Lemma~\ref{lem:1} and the regularity assumptions~\eqref{cond:fEB}, we obtain that
\begin{align}\nonumber
\mathcal{A}_3  &\le C\int_{\mathbb{R}^3} \phi_{\varepsilon}(y)|y|^{\theta+\kappa} \|E\|_{\mathcal{L}^{\infty}\mathcal{W}^{\kappa,q}_3} \|\nabla_{\varsigma}u^{\sigma}(t,x,\varsigma)\|_{\mathcal{L}^{1,p}\mathcal{W}^{\theta,p}} dy,
\end{align}
which yields from~\eqref{eq:reg_gradf} in Lemma~\ref{lem:1a} that
\begin{align}\label{eq:estTE1}
\mathcal{A}_3  \le C \varepsilon^{\theta+\kappa} \sigma^{\theta-1} \|E\|_{\mathcal{L}^{\infty}\mathcal{W}^{\kappa,q}_3}\ \omega_u(\varepsilon,\sigma).
\end{align}
Thanks to H{\"o}lder inequality and~\eqref{eq:reg_f} and~\eqref{eq:cor_b}, the second term on the right hand side of~\eqref{eq:TE_Leb2} can be estimated as
\begin{align} \label{eq:estTE3}
\mathcal{A}_4  \le \|E-E^{\varepsilon}\|_{\mathcal{L}^{\infty,q}_3} \|\nabla_{\varsigma} u^{\sigma} - (\nabla_{\varsigma} u^{\sigma})^{\varepsilon}\|_{ \mathcal{L}^{1,p}_6}  \le C \varepsilon^{\theta+\kappa} \sigma^{\theta-1} \|E\|_{\mathcal{L}^{\infty}\mathcal{W}^{\kappa,q}_3} \ \omega_u(\varepsilon,\sigma).
\end{align}
From what have already been proved in~\eqref{eq:TE_Leb2}, \eqref{eq:estTE1} and \eqref{eq:estTE3}, we get
\begin{align}\label{eq:estTE}
\|\nabla_{\varsigma} \cdot T_E\|_{\mathcal{L}^{1,p,r}}  \le C \varepsilon^{\theta + \kappa} \sigma^{\theta-1} \|E\|_{\mathcal{L}^{\infty}\mathcal{W}^{\kappa,q}_3}\  \omega_u(\varepsilon,\sigma).
\end{align}
We next consider the term $T_B$ given by~\eqref{eq:TB}, which can be decomposed as
\begin{align} \label{eq:TB0}
T_B  = T_{B1}+T_{B2}+ T_{B3}, \ \mbox{ where }
\end{align}
\begin{align} \nonumber
T_{B1} & :=  \int\,  dX \ \phi_{\gamma}(\nu)\phi_{\varepsilon}(y)\phi_{\sigma}(\upsilon) [(v(\varsigma - \upsilon)-v(\varsigma)) \times B(t-\nu,x-y)] u(t-\nu, x-y,\varsigma - \upsilon) , \\ \nonumber 
T_{B2} & :=  v \times ((Bu^{\sigma})^{\gamma,\varepsilon}-B^{\gamma,\varepsilon}(u^{\sigma})^{\gamma,\varepsilon}), \mbox{ and } T_{B3}  :=  (v-v^{\sigma})\times B^{\gamma,\varepsilon}(u^{\sigma})^{\gamma,\varepsilon}.
\end{align}
Here we denote by $ \int\,  dX = \int_0^T d\nu \int_{\mathbb{R}^3}dy \int_{\mathbb{R}^3} d\upsilon$ for simplicity of notations. The first term of~\eqref{eq:TB0} can be decomposed and rewritten as follows
\begin{align} \nonumber
 \nabla_{\varsigma}  \cdot T_{B1} & = \int\,  dX \ \phi_{\gamma}(\nu)\phi_{\varepsilon}(y) \nabla_{\upsilon}\phi_{\sigma}(\upsilon) \\ \nonumber 
& \qquad \qquad \cdot [(v(\varsigma - \upsilon)-v(\varsigma)) \times B(t-\nu,x-y)]  u(t-\nu, x-y,\varsigma) \\ \nonumber 
& + \int\,  dX \  \phi_{\gamma}(\nu)\phi_{\varepsilon}(y)\nabla_{\upsilon}\phi_{\sigma}(\upsilon) \cdot [(v(\varsigma - \upsilon)-v(\varsigma)) \times B(t-\nu,x-y)] \\ \label{eq:TBI} 
& \qquad \qquad \cdot[u(t-\nu, x-y,\varsigma - \upsilon) - u(t-\nu, x-y,\varsigma)] =: I_1 + I_2.
\end{align}
Integrating by parts and since $\nabla_{\upsilon} \cdot [(v(\varsigma - \upsilon)-v(\varsigma)) \times B(t-\nu,x-y)]=0$, one observes that the first term also vanishes:
\begin{multline}\label{eq:I1} 
I_1 = \int\,  dX \ \phi_{\gamma}(\nu)\phi_{\varepsilon}(y) \phi_{\sigma}(\upsilon) \\  \nabla_{\upsilon} \cdot [(v(\varsigma - \upsilon)-v(\varsigma)) \times B(t-\nu,x-y)]  u(t-\nu, x-y,\varsigma) = 0,
\end{multline}
and by H{\"o}lder inequality and estimate~\eqref{eq:est_v} in Lemma~\ref{lem:1b}, it is easy to obtain that
\begin{multline*}
\|I_2\|_{\mathcal{L}^{1,p,r}} \le 2 \int\,  dX \ \phi_{\gamma}(\nu)\phi_{\varepsilon}(y)|\nabla_{\upsilon}\phi_{\sigma}(\upsilon)||\upsilon|  \|B\|_{\mathcal{L}^{\infty,q}_3}   \\ 
\|u(t-\nu, x-y,\varsigma - \upsilon) - u(t-\nu, x-y,\varsigma)\|_{\mathcal{L}^{1,p}_6}. 
\end{multline*}
We then apply the estimate~\eqref{eq:f_w} in Lemma~\ref{lem:1}, the restriction property for Sobolev spaces $W^{\theta,p}(\mathbb{R}^n)$ and regularity assumptions~\eqref{cond:fEB}, it deduces from the above inequality that
\begin{align} \nonumber
\|I_2\|_{\mathcal{L}^{1,p,r}}  &\le C \int_{\mathbb{R}^3} |\nabla_{\upsilon}\phi_{\sigma}(\upsilon)||\upsilon|^{\theta+1} d\upsilon  \|B\|_{\mathcal{L}^{\infty,q}_3}{\|u\|_{\mathcal{L}^{1,p}\mathcal{W}^{\theta,p}}}  \\  \label{eq:I2} 
&\le C \sigma^{\theta}  \|B\|_{\mathcal{L}^{\infty}\mathcal{W}^{\kappa,q}_3} \|u\|_{\mathcal{L}^{1}\mathcal{W}^{\theta,p}_6}.
\end{align}
It follows easily that from \eqref{eq:TBI},  \eqref{eq:I1} and~\eqref{eq:I2}, one has
\begin{align}\label{eq:estTB1}
\|\nabla_{\varsigma}  \cdot T_{B1} \|_{\mathcal{L}^{1,p,r}} \le C \sigma^{\theta}  \|B\|_{\mathcal{L}^{\infty}\mathcal{W}^{\kappa,q}_3} \|u\|_{\mathcal{L}^{1}\mathcal{W}^{\theta,p}_6}.
\end{align}
To estimate $\nabla_{\varsigma} \cdot T_{B2}$, we can now proceed analogously to what we have obtained in \eqref{eq:estTE} for $\nabla_{\varsigma} \cdot T_{E}$, giving
\begin{align} \label{eq:estTB2}
\|\nabla_{\varsigma} \cdot T_{B2}\|_{\mathcal{L}^{1,p,r}}  \le C \varepsilon^{\theta + \kappa} \sigma^{\theta-1} \|B\|_{\mathcal{L}^{\infty}\mathcal{W}^{\kappa,q}_3} \  \omega_u(\varepsilon,\sigma).
\end{align}
H{\"o}lder inequality is used repeatedly to obtain
\begin{align*}
\|\nabla_{\varsigma} \cdot T_{B3}\|_{\mathcal{L}^{1,p,r}}  \le |v-v^{\sigma}| \|B^{\gamma,\varepsilon}\|_{\mathcal{L}^{\infty,q}_3} \|\nabla_{\varsigma} u^{\gamma,\sigma,\varepsilon}\|_{ \mathcal{L}^{1,p}_6}.
\end{align*}
Applying estimate~\eqref{eq:est_v_del} in Lemma~\ref{lem:1b}  and Lemma~\ref{lem:1} to this inequality, we have
\begin{align}\label{eq:estTB3}
\|\nabla_{\varsigma} \cdot T_{B3}\|_{\mathcal{L}^{1,p,r}}  \le C \sigma \|B^{\gamma,\varepsilon}\|_{\mathcal{L}^{\infty,q}_3} \|\nabla_{\varsigma} u^{\gamma,\sigma,\varepsilon}\|_{ \mathcal{L}^{1,p}_6}  
 \le C \sigma^{\theta} \|B\|_{\mathcal{L}^{\infty}\mathcal{W}^{\kappa,q}_3} \|u\|_{\mathcal{L}^{1}\mathcal{W}^{\theta,p}_6}.
\end{align}
Gathering estimates \eqref{eq:estTB1}-\eqref{eq:estTB3}, we obtain from \eqref{eq:TB0} that
\begin{align} \label{eq:estTB}
\|\nabla_{\varsigma} \cdot T_{B}\|_{\mathcal{L}^{1,p,r}}  \le C \|B\|_{\mathcal{L}^{\infty}\mathcal{W}^{\kappa,q}_3} \left( \varepsilon^{\theta + \kappa} \sigma^{\theta-1}  \  \omega_u(\varepsilon,\sigma)  +  \sigma^{\theta}   \|u\|_{\mathcal{L}^{1}\mathcal{W}^{\theta,p}_6}\right).
\end{align}
Finally, by estimates \eqref{eq:estTE} and \eqref{eq:estTB}, we obtain \eqref{eq:Cfl} from~\eqref{eq:TE_TB} and therefore, the proof of Lemma is then complete.
\end{proof}

With aid of Lemma \ref{lem:2} and lemmas presented in the preceding section, we readily prove the main theorem. 

\begin{proof}[Proof of Theorem~\ref{theo:main}]
We firstly use the notation $\int\,  dX = \int_0^T dt \int_{\mathbb{R}^3}dx \int_{\mathbb{R}^3} d\varsigma$ again for simplicity, the weak formulation for the Vlasov equation~\eqref{eq:sys1} reads
\begin{align}\label{eq:weakform}
 \int\,  dX \ u(\partial_t \varphi + v\cdot \nabla_x \varphi + \mathcal{F} \cdot \nabla_{\varsigma}\varphi) = 0, \quad \forall \varphi \in \mathcal{D}((0,T)\times \mathbb{R}^6),
\end{align}
where $\mathcal{F} = E + v \times B$ denotes the Lorentz force field. We remark that integrals in~\eqref{eq:weakform} are finite since for DiPena-Lions weak solutions in~\cite{Diperna}, it is known that $u \in \mathcal{L}^{\infty,2}_6$ and $E, B \in \mathcal{L}^{\infty,2}_3$. For every positive numbers $\gamma$, $\varepsilon$ and $\sigma$, let us take the test function in~\eqref{eq:weakform} as
\begin{align}
\label{eq:testfunction}
\varphi =  (\mathcal{G}'(u^{\gamma,\varepsilon,\sigma})\psi)^{\gamma,\varepsilon,\sigma} \in \mathcal{D}((0,T)\times \mathbb{R}^6),
\end{align}
with $\psi \in \mathcal{D}((0,T)\times \mathbb{R}^6)$ and $\mathcal{G} \in \mathcal{C}^1(\mathbb{R}^+;\mathbb{R}^+)$.
By using Lemma~\ref{lem:1} and successive integrations by parts, we obtain from \eqref{eq:weakform} and \eqref{eq:testfunction} that
\begin{multline*}
 \int\,  dX \ \mathcal{G}(u^{\gamma,\varepsilon,\sigma})\left(\partial_t \psi + v^{\sigma}\cdot \nabla_x \psi + \mathcal{F}^{\gamma,\varepsilon,\sigma}\cdot \nabla_{\varsigma} \psi\right) \\ + \psi \mathcal{G}'(u^{\gamma,\varepsilon,\sigma}) \left[\nabla_x \cdot \left((vu)^{\gamma,\varphi,\sigma}-v^{\sigma}u^{\gamma,\varepsilon,\sigma}\right) + \nabla_{\varsigma} \cdot\left((\mathcal{F}u)^{\gamma,\varepsilon,\sigma}-\mathcal{F}^{\gamma,\varepsilon,\sigma}u^{\gamma,\varepsilon,\sigma}\right)\right]  = 0, 
\end{multline*}
for all $\psi \in \mathcal{D}((0,T)\times \mathbb{R}^6)$. We now establish the renormalized Vlasov equation~\eqref{eq:renorl_property}. Using assumptions \eqref{cond:fEB}, Lemma~\ref{lem:1},~\ref{cor:1} and \ref{lem:2}, we obtain that
\begin{multline}\label{eq:weakform3}
\left|\int\,  dX \ \mathcal{G}(u^{\gamma,\varepsilon,\sigma})\left(\partial_t \psi + v^{\sigma}\cdot \nabla_x \psi + \mathcal{F}^{\gamma,\varepsilon,\sigma}\cdot \nabla_{\varsigma} \psi\right)\right| \\ \le C\left(\varepsilon^{\theta-1}\sigma^{\theta + 1} + \varepsilon^{\theta+\kappa} \sigma^{\theta-1}\right)\omega_u(\varepsilon,\sigma) + C\sigma^{\theta}, 
\end{multline}
where the function $\omega$ given in~\eqref{def:omega} and the constant $C$ depends on $\|u\|_{\mathcal{L}^1\mathcal{W}^{\theta,p}_6}$, $\|B\|_{\mathcal{L}^{\infty}\mathcal{W}^{\kappa,q}_3}$, $\|E\|_{\mathcal{L}^{\infty}\mathcal{W}^{\kappa,q}_3}$, $\mathcal{G}$ and $\psi$. We see that
$$ \varepsilon^{\theta-1}\sigma^{\theta + 1} + \varepsilon^{\theta+\kappa} \sigma^{\theta-1} = \varepsilon^{\theta-1}\sigma^{\theta - 1} \left(\sigma^2 + \varepsilon^{\kappa +1}\right).  $$
 Therefore, to balance contributions coming from two terms on the right hand side of~\eqref{eq:weakform3}, we may choose $\sigma^2 = \varepsilon^{\kappa+1}$, which guarantees that 
\begin{align*}
\left| \int\,  dX \ \mathcal{G}(u^{\gamma,\varepsilon,\sigma})\left(\partial_t \psi + v^{\sigma}\cdot \nabla_x \psi + \mathcal{F}^{\gamma,\varepsilon,\sigma}\cdot \nabla_{\varsigma} \psi\right)\right|  \le C\left(\varepsilon^{\frac{\theta \kappa + \kappa + 3\theta-1}{2}}\omega_u(\varepsilon,\sigma) + \sigma^{\theta}\right). 
\end{align*}
Under our general assumption $\theta \kappa + \kappa + 3\theta-1 \ge 0$, we deduce that
\begin{align}\label{eq:weakform4}
\left| \int\,  dX \ \mathcal{G}(u^{\gamma,\varepsilon,\sigma})\left(\partial_t \psi + v^{\sigma}\cdot \nabla_x \psi + \mathcal{F}^{\gamma,\varepsilon,\sigma}\cdot \nabla_{\varsigma} \psi\right)\right|  \le C\left(\omega_u(\varepsilon,\sigma) + \sigma^{\theta}\right).
\end{align}
Thanks to Lemma~\ref{lem:omega} and $\theta \in (0,1)$, the right hand side of~\eqref{eq:weakform4} vanishes as $(\varepsilon,\sigma)$ goes to 0. So we obtain the renormalization property~\eqref{eq:renorl_property} of the Vlasov equation. Finally, the local and global entropy conservation laws~\eqref{eq:local_law1}, \eqref{eq:global_law} can be established by the same method as in~\cite{Bardos} under assumption $\theta \kappa + \kappa + 3\theta-1 \ge 0$. 
\end{proof}


\end{document}